\documentclass[12pt]{amsart}

\usepackage{amssymb}
\usepackage{amsmath}
\usepackage{amsfonts}
\usepackage{bm} % \bm for bold math symbols
\usepackage[all]{xy} % For commutative diagrams
\usepackage{mathrsfs,enumitem}
\usepackage[utf8]{inputenc}
\usepackage[usenames,dvipsnames,svgnames,table]{xcolor}
\usepackage{ulem} % for strikethough of text with command \sout

\newcommand{\disk}{\ensuremath{\mathbb{D}} } % unit disk
\newcommand{\riem}{\Sigma}
\newcommand{\Oqc}{\mathcal{O}^{\mathrm{qc}}} % Oqc

\newcommand{\sphere}{\overline{\mathbb{C}}}
\newcommand{\Gr}{\operatorname{Gr}}
\newcommand{\hatGr}{\widehat{\Gr}}
\newcommand{\Grfull}{\operatorname{\mathbf{Gr}}}
\newcommand{\Grfullf}{\Grfull(\bm{f})}
\newcommand{\hatGrfull}{\widehat{\Grfull}}
\newcommand{\hatGrfullf}{\hatGrfull(\bm{f})}
\newcommand{\If}{\operatorname{\mathbf{I}}_{\bm{f}}}
\newcommand{\hatIf}{\hat{\operatorname{\mathbf{I}}}_{\bm{f}}}

\theoremstyle{plain}
        \newtheorem{theorem}{Theorem}[section]
        \newtheorem{lemma}[theorem]{Lemma}

\theoremstyle{definition}
        \newtheorem{definition}[theorem]{Definition}

\theoremstyle{remark}
    \newtheorem{remark}[theorem]{Remark}

\numberwithin{equation}{section}
\numberwithin{figure}{section}

\usepackage{fullpage}

\title{A Model of the Teichm\"uller space of genus-zero bordered surfaces by period maps}
\author{David Radnell}
\author{Eric Schippers}
\author{Wolfgang Staubach}

\begin{document}

\begin{abstract}
 We consider Riemann surfaces $\riem$ with $n$ borders homeomorphic to $\mathbb{S}^1$ and no handles.  Using generalized Grunsky operators, we define a period mapping from the infinite-dimensional Teichm\"uller space of surfaces of this type into the unit ball in the linear space of operators on an $n$-fold direct sum of Bergman spaces of the disk.  We show that this period mapping is holomorphic and injective.
\end{abstract}
\thanks{D. Radnell acknowledges the support of the Academy of Finland's project ``Algebraic structures and random geometry of stochastic lattice models". E. Schippers and W. Staubach author are grateful for the financial support from the Wenner-Gren Foundations. E. Schippers is also partially supported by the National Sciences and Engineering Research Council of Canada. }
\maketitle
\begin{section}{Introduction}
\begin{subsection}{Introduction}
 The classical period mapping takes compact Riemann surfaces of genus $g$ into the Siegel upper half-plane, which consists of symmetric $g \times g$ matrices with positive definite imaginary part.  It is a classical fact that this map is holomorphic \cite{Nagbook}.

 S. Nag \cite{NagBulletin} and S. Nag and D. Sullivan \cite{NagSullivan} constructed a period mapping of the universal Teichm\"uller space $T(\disk^+)$, where $\disk^+ = \{z\,:\, |z| <1 \}$. This period map takes the infinite-dimensional Teichm\"uller space into the Siegel disk of bounded operators $T$ on the Dirichlet space of the disk satisfying $\| T \| <1$.  This is an alternate formulation of the Siegel upper half-plane of operators with positive-definite imaginary part.   L. Takhtajan and L.-P. Teo  \cite{TakhtajanTeo_memoirs} later showed, remarkably, that the period mapping is in fact the Grunsky operator of univalent function theory \cite{Duren,Pommerenkebook}, and gave the first complete proof that the period mapping is holomorphic.

 In this paper, we generalize the period mapping to the case of the Teichm\"uller space of genus-zero surfaces with $n$ closed non-overlapping disks removed. The period mapping takes the Teichm\"uller space of this type into the direct product of the Teichm\"uller space of genus-zero surfaces with $n$ punctures and a space of bounded operators on an $n$-fold sum of Bergman spaces of the disk.  The portion mapping into the Teichm\"uller space of punctured surfaces can of course be represented by period matrices using the classical method.

 Our construction uses a generalized Grunsky operator, which was shown by the authors to be bounded by one \cite{RSS_Dirichlet_full}, and thus lies in a kind of Siegel disk.  We show that this mapping is holomorphic.  The separation of the period mapping into a finite-dimensional part, involving compact surfaces with punctures, and an infinite-dimensional part consisting of bounded operators on direct sums of Bergman spaces, relies on a fiber structure of Teichm\"uller space discovered by D. Radnell and E. Schippers \cite{RS_fiber}.  Holomorphicity of this fibration, and a resulting new set of complex coordinates \cite{RS_fiber}, plays a key role in our proof of holomorphicity of the period map.  The demonstration of this was accomplished using
 a variational technique of Radnell \cite{Radnell_thesis} which was obtained by modifying that of F. Gardiner and M. Schiffer \cite{Gardiner,Nagbook}.
 \end{subsection}
\begin{subsection}{Bergman spaces of one-forms}  We establish some notation for Bergman spaces.
 Let $\Omega$ be a domain in $\sphere$.
 Define $A^2(\Omega)_{\text{harm}}$ to be the set of harmonic one-forms $\alpha$ on $\Omega$
 which are $L^2$ in the sense that
 \[   \| \alpha \|^2 = \frac{i}{2} \iint_\Omega \alpha \wedge \overline{\alpha} < \infty.  \]
 We will call this the harmonic Bergman space.  It has a natural inner product given by
 \begin{equation} \label{eq:inner_product}
   (\alpha,\beta) = \frac{i}{2} \iint_\Omega \alpha \wedge \overline{\beta}.
 \end{equation}

 The subset of $A^2(\Omega)_{\text{harm}}$ consisting of holomorphic one-forms, is the Bergman space which is denoted by $A^2(\Omega)$.
 We will represent one-forms in the Bergman space by functions.  That is, if $\alpha$ is a one-form in $A^2(\Omega)$,
 then in $\Omega \backslash \{\infty \}$ it has a unique expression $h(z)\,dz$.  In a neighbourhood of
 $\infty$, using the chart $w \mapsto 1/w$, $\alpha$ has the expression $\alpha = - w^{-2}h(1/w)\, dw$.

 The condition that
 $\alpha = h(z)\, dz\in L^2$ can then be expressed as follows.  For some $r>1$, set $U = \{ z\,:\, |z| <r \}$
 and $V = \{ z\,:\, |z|>1/r \} \cup \{ \infty \}$.  Then $\alpha$ is in $A^2(\Omega)$ if and only if
 \begin{equation} \label{eq:integral_conditions_infinity}
  \iint_{\Omega \cap U} |h(z)|^2  dA_z < \infty \ \ \text{and} \ \
 \iint_{1/(\Omega \cap V)} |w^{-2} h(1/w)|^2 dA_w < \infty.
 \end{equation}
 Here we use $dA_z$ as an abbreviation for $ (d\bar{z} \wedge dz)/2i$, and $1/D$ means
 $\{z \in \sphere \,:\, 1/z \in D \}$.  If both conditions are satisfied then
 \[  \frac{i}{2} \iint_\Omega \alpha \wedge \overline{\alpha} = \iint_{\Omega \backslash \{\infty \}}
       |h(z)|^2 dA_z.  \]
 We will abbreviate the expression for the right hand integral by
 \[   \iint_{\Omega} |h(z)|^2  dA_z, \]
 though the reader should keep in mind the implicit condition on $h$ at $\infty$.

If $\Omega$ is simply connected, then every $\alpha \in A^2(\Omega)_{\text{harm}}$ has a unique decomposition $\alpha = h(z)\, dz + \overline{g(z)} \,d\bar{z}$ for some holomorphic functions $g$ and $h$ in $A^2(\Omega)$.
 That is,
 \[  A^2(\Omega)_{\text{harm}} = A^2(\Omega) \oplus \overline{A^2(\Omega)}.  \]
 It is easily checked that this decomposition is orthogonal with respect to the inner product (\ref{eq:inner_product}).
 Thus we have that
 \[ \| h(z) dz + \overline{g(z)} d\bar{z} \|^2= \iint_\Omega \left( |h(z)|^2 + |g(z)|^2 \right) \,dA_z.   \]

 We will also consider the space of exact one-forms in the harmonic or holomorphic Bergman space,
 which we denote by
 \[  A^2_e(\Omega) = \{ \alpha \in A^2(\Omega) \,:\, \alpha = dH \ \ \text{for some holomorphic} \  H  \}  \]
 and similarly for $A^2_e(\Omega)_{\text{harm}}$.
 In $A^2_e(\Omega)$, if we express $\alpha = h(z)\,dz$, there is some holomorphic function $H$ with domain $\Omega$ such that $H'(z) = h(z)$.  Holomorphicity on $\Omega$ means that $H$ is holomorphic on $\Omega \cap U$
 and $H(1/z)$ is holomorphic on $1/(\Omega \cap V)$.  If $\infty \in \Omega$, this implies in particular that  $H$ has a finite limit as $z \rightarrow \infty$; equivalently, $H$ is continuous in the sphere topology.
 Note that the decomposition of $A^2(\Omega)_{\text{harm}}$ restricts to a decomposition
 $A^2_e(\Omega)_{\text{harm}} = A^2_e(\Omega) \oplus \overline{A^2_e(\Omega)}$, when $\Omega$ is simply connected .

 Up to constants, the exact Bergman spaces are thus each isometric with a Dirichlet space.  The  harmonic Dirichlet space $\mathcal{D}(\Omega)_{\text{harm}}$ is the space of harmonic
 functions $H$ such that
 \begin{equation} \label{eq:Dirichlet_seminorm}
  \frac{i}{2} \iint_{\Omega} dH \wedge \overline{dH} <\infty.
 \end{equation}
 The Dirichlet space of holomorphic functions is denoted $\mathcal{D}(\Omega)$.
 For $p \in \Omega$, $\mathcal{D}_p(\Omega)_{\text{harm}}$ denotes the subset of $\mathcal{D}(\Omega)_{\text{harm}}$ whose elements vanish at $p$, and similarly for $\mathcal{D}_p(\Omega)$.
For simply connected domains $\Omega$, the elements of the harmonic Dirichlet space have a decomposition $H = F + \overline{G}$ where $F$
 and $G$ are holomorphic, so that
 we can write
 \[ \mathcal{D}_p(\Omega)_{\text{harm}} = \mathcal{D}_p(\Omega) \oplus \overline{\mathcal{D}_p(\Omega)}.    \]
Note that in $\mathcal{D}(\Omega)_{\text{harm}}$ the decomposition is not unique because constants are both holomorphic
 and anti-holomorphic.
 We  have the isometry
 \begin{align}  \label{eq:definition_isometry}
  d:\mathcal{D}_p(\Omega)_{\text{harm}} & \longrightarrow A^2_e(\Omega)_{\text{harm}} \nonumber \\
  H & \longmapsto dH.
 \end{align}
 The decompositions of  $\mathcal{D}_p(\Omega)_{\text{harm}}$ and $A^2_e(\Omega)_{\text{harm}}$ commute with this isometry.

 Let $f:D_1 \rightarrow D_2$ be a biholomorphism between two domains $D_1$, $D_2$ in    $\sphere$.  We have a pull-back operator defined by
 \begin{align*}
  \hat{\mathcal{C}}_f: A^2(D_2)_{\text{harm}} & \longrightarrow A^2(D_1)_{\text{harm}} \\
  h(z)\, dz + \overline{g(z)}\,d \bar{z}  & \longmapsto h \circ f(z) \cdot f'(z) \,dz + \overline{g \circ f(z) \cdot f'(z)}    \,  d\bar{z}
 \end{align*}
 This is clearly an isometry.  Furthermore, $\hat{\mathcal{C}}_f$ restricts to an isometry
 from $A^2_e(D_2)_{\text{harm}}$ to $A^2_e(D_1)_{\text{harm}}$.  It also restricts to
 an isometry from  $A^2(D_2)$ to $A^2(D_1)$ and from $A^2_e(D_2)$ to $A^2_e(D_1)$.

 The composition operator
 \begin{align*}
   \mathcal{C}_f:\mathcal{D}_p(D_2)_{\text{harm}} & \longrightarrow \mathcal{D}_{f(p)}(D_1)_{\text{harm}} \\
    H & \longmapsto H \circ f
 \end{align*}
  is also an isometry, and we have that:
  \[  d \circ \mathcal{C}_f = \hat{\mathcal{C}}_f \circ d,  \]
  which incidentally motivates the notation $\hat{\mathcal{C}}_f$.

  \begin{remark}[Notation]
  Throughout the paper, operators without
  hats act on functions and operators with hats act on one-forms. We shall also denote the closure of a set $A$ by $A^{\mathrm{cl}}$, and its interior by $A^{\mathrm{int}}$.
  \end{remark}

 In the remainder of this paper, we will usually identify the elements $\alpha = h(z)\,dz$ of the holomorphic  Bergman space
 with the function $h(z)$, except when emphasizing the fact that the elements are one-forms.  The function is always written as a function of the standard coordinate $z$ in $\mathbb{C} \subset \sphere$ rather than as a function of a coordinate at $\infty$.

 We will not be directly working with Dirichlet spaces in this paper. They will be used only to apply results of
  the authors \cite{RSS_Dirichlet_full} for Dirichlet spaces to Bergman spaces, through the use of the isometry
 (\ref{eq:definition_isometry}).  These results
 involve a ``reflection'' of harmonic Dirichlet functions in quasidisks, obtained by extending to the
 boundary of the quasidisks, and then extending them to the complementary quasidisk. One may summarize the situation as follows: in the present paper, the use of one-forms creates a clearer geometric picture, whereas in the paper \cite{RSS_Dirichlet_full},
 the use of functions created a more clear analytic picture.
\end{subsection}
\end{section}

\begin{section}{Grunsky map for multiply-connected domains} \label{se:reflection}
\begin{subsection}{The generalized Faber and Grunsky operators}
 In this section we define certain generalizations of a Faber operator and the Grunsky operator
 to multiple maps with non-overlapping images.
 First we define the Faber operator and Grunsky operator associated with a single conformal map.
 For the concept of a Faber operator see P. Suetin \cite{Suetin_monograph}; for the Grunsky  operator see for example \cite{BergmanSchiffer,Duren,Pommerenkebook}.

 Let $\Gamma$ be a Jordan curve not containing $\infty$, and let $\Omega^+$ be the bounded component
 of the complement of $\Gamma$ in $\sphere$, and $\Omega^-$ be the other complementary component.
   Let
 \[  \disk^+ = \{z\,:\, |z|<1 \} \quad \text{and} \quad  \disk^- = \{z \,:\, |z| >1 \} \cup \{ \infty \}.  \]
 Let $f:\disk^+ \rightarrow \Omega^+$ be a conformal map.
 Following \cite{RSS_Dirichlet_full, SchippersStaubach_jump}, we define the operators
 \[
 P(\Omega^\pm) : \mathcal{D}_{\text{harm}}(\Omega^+) \longrightarrow \mathcal{D}(\Omega^\pm)
 \]
 by  \[  [P(\Omega^\pm) h](z) = \pm \lim_{r \nearrow 1} \frac{1}{2 \pi i} \int_{f(C_r)}
    \frac{h(\zeta)}{\zeta -z } \,d\zeta, \quad z \in \Omega^\pm,  \]
 where $C_r$ is the circle $\{w\,:\, |w|=r \}$ traced counter-clockwise. Furthermore, define the map (the Faber operator)
 \begin{align*}
  \text{I}_f: \mathcal{D}_\infty(\disk^-) & \longrightarrow \mathcal{D}_\infty(\Omega^-)  \\
  h & \longmapsto P(\Omega^-) \mathcal{C}_{f^{-1}} \mathbf{R} h
 \end{align*}
 where $\mathbf{R}:\mathcal{D}_\infty(\disk^-) \rightarrow \mathcal{D}_0(\disk^+)$ is
 given by $\mathbf{R} h(z) = h(1/\bar{z})$.

 The limiting integral is necessary since Jordan curves are of course not in general
 rectifiable.
 The operators $P(\Omega^\pm)$ were shown to be well-defined maps which are bounded with respect to the Dirichlet semi-norm (\ref{eq:Dirichlet_seminorm}).
 It was also shown in \cite{SchippersStaubach_jump} that the Faber operator is an isomorphism precisely for quasicircles.
 This remarkable result is originally due to Y. Shen \cite{Shen_Faber}, with a somewhat different formulation
 of the operator; closely related results for convergence of Faber series on quasidisks were obtained
 by A. \c{C}avu\c{s} \cite{Cavus}.

 We now consider the multiply-connected case.
 The following notation will be in force for the remainder of the paper.  Let $\riem \subset \sphere$
 be a multiply-connected domain, which is bounded by $n$ non-overlapping quasicircles $\Gamma_i$,
 $i=1,\ldots,n$.  We assume that $\infty \in \riem$.  This normalization is a matter of convenience, and will be removed shortly.
 Let $\Omega_i^+$ denote the component of the complement of $\Gamma_i$ in $\sphere$ which does not intersect $\riem$, and let $\Omega_i^-$ denote the other component of the complement of $\Gamma_i$.  For each $i$, $\Omega_i^-$ contains $\riem$; in fact
 \[  \riem = \bigcap_{i=1}^n \Omega_i^-.  \]
  We will also fix points $p_i \in \Omega_i^+$ for $i=1,\ldots,n$.

% \begin{remark}
% We can characterize the difference
% between $A^2_e(\riem)$ and $A^2(\riem)$ as follows.  Let $H(p_1,\ldots,p_n)$ denote
% the vector space of holomorphic one-forms on $\sphere \backslash \{p_1,\ldots,p_n \}$, which
% have removable singularities or poles of order no more than one at $p_1,\ldots,p_n$.
% We then have the direct sum decomposition
% \[  A^2(\riem) = A^2_e(\riem) \oplus H(p_1,\ldots,p_n).     \]
% To see this, let $\alpha$ be an arbitrary one-form in $A^2(\riem)$.  Let $c_i$ be the period
% of $\alpha$ around the boundary $\Gamma_i$ (integrating around a smooth oriented curve in $\riem$
% which is homotopic to $\Gamma_i$). Note that $\sum c_i =0$ since the boundary $\partial \riem$ is
% null-homotopic in $\sphere \backslash \{p_1,\ldots,p_n \}$.
% Thus
% \[  \alpha_e = \alpha - \sum_{i=1}^n \frac{c_i}{z-p_i}  \]
% has vanishing periods in $\riem$ and $\sum_{i=1}^n \frac{c_i}{z-p_i} \in H(p_1,\ldots,p_n)$
% since $\sum_{i=1}^n {c_i}=0$ (so that there is no pole at $\infty$).  This decomposition is unique: given two such decomposition $\alpha = \alpha_e +
% \beta$ and $\alpha = \alpha_e' + \beta'$ we would have that $\beta'-\beta$ has zero periods and hence
% is a holomorphic one-form on $\sphere$, and thus is zero.
% \end{remark}

 For $i=1,\ldots,n$, fix conformal maps $f_i:\disk^+\rightarrow \Omega^+_i$ such that
 $f_i(0)=p_i$.  Let $\bm{f}=(f_1,\ldots,f_n)$.  In \cite{RSS_Dirichlet_full} the following generalized Faber operator was defined:
 \begin{align*}
  \If:\bigoplus^n \mathcal{D}_\infty(\disk^-) & \longrightarrow \mathcal{D}_\infty(\riem) \\
  (h_1,\ldots,h_n) & \longmapsto \sum_{i=1}^n \text{I}_{f_i} h_i.
 \end{align*}
 It was shown in \cite{RSS_Dirichlet_full}
 that this is an isomorphism.  The generalized Grunsky operator was
 also defined:
 \[   \Grfullf = \left( \text{P}_0(\disk^+)\mathcal{C}_{f_1} \If,
 \ldots, \text{P}_0(\disk^+) \mathcal{C}_{f_n} \If  \right): \bigoplus^n \mathcal{D}_\infty(\disk^-)  \longrightarrow \bigoplus^n
  \mathcal{D}_0(\disk^+) \]
 where
 \[  [P_0(\disk^+) H ](z)= [P(\disk^+) H](z) - [P(\disk^+) H](0).  \]
 The blocks of this matrix (taking the $i$th component to the $j$th component
 of the direct sum) are denoted $\Gr_{ji}$. Note that this block depends only on $f_i$ and $f_j$ but for notational convenience we mostly write $\Gr_{ji}$ instead of $\Gr_{ji}(f_i, f_j)$. Further technical work was required to make sense of the composition $\mathcal{C}_{f_j} \If$; this was
 accomplished in \cite{RSS_Dirichlet_full} and publications cited therein.  Essentially, one may think of
 the composition operator as acting on boundary values of harmonic functions.  In
 this paper, we will derive an equivalent integral formula and work directly with that.

 Generalized Grunsky operators for non-overlapping mappings were considered by J. A. Hummel \cite{Hummel}.  They
 are also considered in Takhtajan and Teo \cite{TakhtajanTeo_memoirs}
 in the case of a pair of non-overlapping maps whose images
 fill the sphere minus a quasicircle (that is, for a conformal welding pair).

 We would like to use the equivalent form of the generalized Grunsky operator on exact
 one-forms rather than functions.
 Let
 \[  \oplus^n d = (d,\ldots,d):\bigoplus^n \mathcal{D}_0(\disk^+) \longrightarrow \bigoplus^n A^2(\disk^+)
       \]
 and similarly define
 \[  \oplus^n d^{-1}: \bigoplus^n A^2(\disk^-)  \longrightarrow \bigoplus^n \mathcal{D}_\infty(\disk^-).  \]
 Thus we may define
 \[  \hatIf = d \circ \If \circ \oplus^n d^{-1}:
    \bigoplus^n A^2(\disk^-) \longrightarrow A^2_e(\riem) \]
 and
 \begin{equation}
 \label{eq:hatGr}
 \hatGrfullf  = \oplus^n d \circ \Grfullf \circ \oplus^n d^{-1}
    : \bigoplus^n A^2(\disk^-) \longrightarrow \bigoplus^n A^2(\disk^+)
  \end{equation}
 with the blocks $\hatGr_{ji}(f_j,f_i)$ similarly being defined as the block components of
 $\hatGrfullf$.  We will abbreviate these blocks as $\hatGr_{ji}$.  In the rest of the paper, we will use the Pythagorean norm on the direct sum $\bigoplus^n A^2(\disk^-)$
 \[  \| (h_1,\ldots,h_n) \|_{\bigoplus^n A^2(\disk^-)}^2 = \sum_{k=1}^n \| h_k\|_{A^2(\disk^-)}^2   \]
 and similarly for $\bigoplus A^2(\disk^+)$.

 \begin{remark}  \label{re:Grunsky_interpretation}
  It can be shown that the graph of the Grunsky operator in $\mathcal{D}_\infty(\disk^+) \oplus  \mathcal{D}_0(\disk^-)$
  is the pull-back of the Dirichlet space $\mathcal{D}_\infty(\riem)$ under
  $(\mathcal{C}_{f_1},\ldots,\mathcal{C}_{f_n})$ \cite{RSS_Dirichlet_full}.
  It follows immediately from the fact that $\If$ is an isomorphism  that $\hatIf$ is also an isomorphism.  Using this fact we can interpret
  the graph of $\hatGrfullf$ as the pull-back of $A^2_e(\riem)$ under
  $\left( \hat{\mathcal{C}}_{f_1},\ldots,\hat{\mathcal{C}}_{f_n} \right)$, so long as we interpret
  $\hat{\mathcal{C}}_{f_i} \hatIf$ as $d\, {\mathcal{C}}_{f_i} \If
  d^{-1}$.  Although we will not make use of this fact in our proofs, it is an important point for interpretation
  of the results of this paper.
 \end{remark}

 \begin{theorem}[\cite{RSS_Dirichlet_full}]  \label{th:Grunsky_bounded}  Let $\riem \subseteq \sphere$ be a domain containing $\infty$, bounded by $n$
  non-intersecting quasicircles $\Gamma_i$, $i=1,\ldots,n$.   Let $\Omega_i^+$ and $\Omega_i^-$ be the bounded and unbounded components
  of the complement of the quasicircle $\Gamma_i$, and let $\bm{f} = (f_1,\ldots,f_n)$
  for conformal maps $f_i: \disk^+ \rightarrow \Omega^+_i$, $i=1,\ldots,n$.  The Grunsky operator $\hatGrfullf$ satisfies $\| \hatGrfullf \|_{\oplus^n A^2(\disk^-) \rightarrow
  \oplus^n A^2(\disk^+)} < 1$.
 \end{theorem}
 \begin{proof}  By a result of \cite{RSS_Dirichlet_full}, the operator norm of
  $\mathbf{Gr}$ is strictly bounded by one.  The claim thus follows from the fact that $d:\mathcal{D}_\infty(\disk^-) \rightarrow A^2(\disk^-)$ and $d:\mathcal{D}_0(\disk^+) \rightarrow A^2(\disk^+)$ are isometries.
 \end{proof}
\end{subsection}

\begin{subsection}{Holomorphicity of $\hatGrfullf$ as a function of $\bm{f}$}
 Here we show that the operator $\hatGrfullf$ is holomorphic as a function of $\bm{f}=(f_1,\ldots,f_n)$.  To do this, certain integral expressions
 for the components of $\hatGrfullf$ are required. First we define the anti-holomorphic reflection
 \begin{align*}
  \hat{\mathbf{R}}:A^2(\disk^-) & \longrightarrow \overline{A^2(\disk^+)} \\
  h(z) \, dz & \longmapsto -  \bar{z}^{-2} h(1/\bar{z})\, d \bar{z}.
 \end{align*}
 This is an anti-isometry by change of variables.
 \begin{theorem} \label{th:integral_representation} Let $\riem$, $\Omega^\pm_i$, $p_i$, and $f_i$ be as above, for $i=1,\ldots,n$.
  We have that for any $i \in \{1,\ldots,n\}$ and $\alpha(z) = h(z)\, dz \in A^2(\disk^-)$
  \[ \hatGr_{ii} h(z) = \frac{1}{\pi} \iint_{\mathbb{D}^+} \left[ \frac{1}{(\zeta - z)^2}
    - \frac{f_i'(\zeta) f_i'(z)}{(f_i(\zeta)-f_i(z))^2} \right] \hat{\mathbf{R}} h(\zeta) \,dA_\zeta.
  \]
Furthermore, for any $i, j \in \{1,\ldots, n\}$ such that $i \neq j$ we have
  \[  \hatGr_{ji} h(z) = \frac{1}{\pi} \iint_{\mathbb{D}^+}  \frac{f_i'(\zeta) f_j'(z)}
    {(f_i(\zeta)-f_j(z))^2}  \, \hat{\mathbf{R}}h(\zeta) dA_\zeta.     \]
 \end{theorem}
 \begin{proof}
  The first claim is \cite[Theorem 4.13]{SchippersStaubach_Grunsky_quasicircle}, and the
  second follows by differentiating \cite[Theorem 4.5]{RSS_Dirichlet_full}.   Differentiating  under the integral sign is justified by the fact that the integrand is absolutely convergent, locally uniformly in $z$.  To see this, observe that since $|f_i(\zeta)-f_j(z)| \geq M$ for $\zeta, z \in \mathbb{D}^+$, the Cauchy-Schwarz inequality yields that for any compact set $K \subset \disk^+$ and for all $z \in K$ one has
 \begin{equation}
 \begin{split}
 \frac{1}{\pi} \iint_{\mathbb{D}^+} \left| \frac{f_i'(\zeta) f_j'(z)}
    {(f_i(\zeta)-f_j(z))^2}  \, \hat{\mathbf{R}}h(\zeta) \right| dA_\zeta
    \leq \frac{|f_j'(z)|}{M^2} \| f_i' \|_{A^2(\disk^+)} \| \hat{\mathbf{R}}h \|_{\overline{A^2(\disk^{+})}} \\ \leq \frac{\Vert f_j'\Vert_{L^{\infty} (K)}}{M^2} \| f_i' \|_{A^2(\disk^+)} \| h \|_{A^2(\disk^-)},
\end{split}
\end{equation}
   where we have also used the fact that $\hat{\mathbf{R}}$ is an isometry.  Since $f_i(\disk^+)$ has finite area the claim follows.
 \end{proof}
 \begin{remark}  Since $\hat{\mathbf{R}}h \in \overline{A^2(\disk^+)}$, we could
  consider the Grunsky operator as a conjugate complex linear operator on $A^2(\disk^+)$;
  see for example S. Bergman and M. Schiffer
  \cite{BergmanSchiffer}.  Inserting the reflection in the circle is natural in our interpretation of the
  Grunsky operator \cite{SchippersStaubach_Grunsky_quasicircle}, and conveniently
  makes the operator complex linear on $A^2(\disk^-)$.
 \end{remark}

 The integral kernels in Theorem \ref{th:integral_representation} are M\"obius invariant, as we now show.
 For any $i, j \in \{1,\ldots,n\}$ (allowing $i =j$), it is easily seen that
   \[   \frac{(T\circ f_i)'(\zeta) \cdot (T \circ f_j)'(z)}{(T \circ f_i(\zeta)- T \circ f_j(z))^2}  =
   \frac{f_i'(\zeta) f_j'(z)}{(f_i(\zeta) - f_j(z))^2}  \]
   for M\"obius transformations $T$ of the form $T(z) = cz$ and $T(z)=z+b$, $b \in \mathbb{C}$,
   $c \in \mathbb{C} \backslash \{0\}$.  For $T(z)=1/z$, we compute
   \[   \frac{(1/f_i)'(\zeta) \cdot (1/f_j)'(z)}{\left( 1/f_i(\zeta) - 1/f_j(z) \right)^2}
      = \frac{f_i'(\zeta) f_j'(z)}{(f_i(\zeta) - f_j(z))^2}.  \]
   Since the group of M\"obius transformations is generated by these two types of transformations, the claim follows.

 Thus, we can define the operator $\hatGrfullf$ for $\bm{f} = (f_1,\ldots,f_n)$ even when one of the quasidisks $f_i(\disk)$ contains $\infty$ in
 its closure by composing $\bm{f}$ with a M\"obius transformation (equivalently, by using the integral expression as a definition).   Note that this also shows that the integral kernel of any
 block $\hatGr_{ji}$ is non-singular
 on $\disk^+ \times \disk^+$, regardless of whether $\infty$ is in the image of $f_j$
 or $f_i$.  With this extension of the definition to general $\bm{f}$, we have now shown the following:
 \begin{theorem}  \label{th:Grunsky_Mobius_invariant}  Let $\riem$ be an open subset of $\sphere$, bounded by $n$ non-overlapping quasicircles $\Gamma_i$.  Otherwise
 let $\Omega_i^\pm$, $p_i$, and $f_i$ be as in Theorem \ref{th:Grunsky_bounded}
  for $i=1,\ldots,n$.  For any M\"obius transformation $T$, denoting $(T \circ f_1,\ldots,
  T \circ f_n)$ by $T \circ \bm{f}$ , we have
  \[  \hatGrfull(T \circ \bm{f}) = \hatGrfullf. \]
  Furthermore the operator norm of $\hatGrfullf$ is strictly less than one.
 \end{theorem}
 \begin{remark}
  The operators $\hatIf$ and $\hat{\mathcal{C}}_{\bm{f}_i}$ also
  extend to the case that $\infty \notin \riem$, so that the interpretation of
  $\hatGrfullf$ of Remark \ref{re:Grunsky_interpretation} continues to hold.
  Since this not necessary for the proof of the main theorem (and indeed is fairly routine)
  we omit it.
 \end{remark}

 We now require some definitions and results of Radnell and Schippers on non-overlapping
 maps into Riemann surfaces with punctures \cite{RSnonoverlapping}.  Punctured Riemann surfaces
 will be denoted with a superscript $P$.
 Let
 \[  A_1^\infty(\disk^+) = \left\{ \psi:\disk^+ \rightarrow \mathbb{C} \,:\, \psi \text{ holomorphic, }
     \| \psi \|_{A_1^\infty(\disk^+)} = \sup_{z \in \disk^+} (1-|z|^2) |\psi(z)| < \infty \right\}.  \]
  Let $\Oqc$ denote the set of injective conformal maps $g:\disk^+ \rightarrow \mathbb{C}$ such that
 $g(0)=0$ and $g$ is quasiconformally extendible to a map from $\sphere$ to $\sphere$.
 The map
 \begin{align} \label{eq:chi_definition} \begin{split}
  \chi:\Oqc & \longrightarrow \mathbb{C} \oplus A^\infty_1(\disk^+) \\
  g & \longmapsto \left( g'(0), g''/g' \right)
 \end{split} \end{align}
 is a bijection onto an open subset of the Banach space $\mathbb{C} \oplus A^\infty_1(\disk^+)$ with
 respect to the direct sum norm by \cite[Theorem 3.1]{RSnonoverlapping}.  Thus $\Oqc$ inherits
 a complex structure by pull-back.
 We also let
 \[  \Oqc(n) =  \{ \bm{g}=(g_1,\ldots,g_n) \,:\, g_i \in \Oqc \text{ for } i=1,\ldots,n \}  \]
 which also has a complex structure obtained by taking the direct sum of $n$ copies of $\mathbb{C} \oplus A^\infty_1(\disk^+)$, again with the direct sum norm.
 Finally,
 \begin{definition}[\cite{RSnonoverlapping}] Let $\riem^P$ be a compact Riemann surface with punctures $p_1,\ldots,p_n$.  We define
  $\Oqc(\riem^P)$ to be the set of $n$-tuples $\bm{f}=(f_1,\ldots,f_n)$ of injective conformal maps
  $f_i:\disk^+ \rightarrow \riem^P$ such that for $i=1,\ldots,n$ the map $f_i$ has a quasiconformal extension to an open neighbourhood
  of the closure of $\disk^+$, $f_i(0)=p_i$, and $f_i(\disk^+)^{\mathrm{cl}} \cap f_j(\disk^+)^{\mathrm{cl}}$ is empty whenever $i \neq j$. We call $\bm{f} \in \Oqc(\riem^P)$ a \textit{rigging} of $\riem^P$.
 \end{definition}
 In this article we are concerned with the special case that $\riem^P =\sphere
 \backslash \{p_1,\ldots,p_n \}$.
 \begin{remark}
 Holomorphic maps and quasiconformal maps between punctured surfaces have unique holomorphic or quasiconformal
 continuations respectively to the compactifications.  We will not distinguish notationally between these
 maps and their extensions.  A punctured surface can be equivalently represented as a compact
 surface with marked points.
 \end{remark}
 \begin{remark}  \label{re:Oqccoords}  The following fact plays an important role ahead.
 In \cite{RSnonoverlapping} we showed that $\Oqc(\riem^P)$ has a natural complex structure in general.
 The local coordinates simplify in the special case that $\riem^P$ is the sphere with $n$ punctures
 $\sphere \backslash \{ p_1,\ldots,p_n \}$.  By \cite[Corollary 3.5]{RSnonoverlapping}, if we choose compact sets $K_i$ such that $p_i$ is in $K_i^{\mathrm{int}}$ for each $i$, then
 \[  V = \{ (f_1,\ldots,f_n) \in \Oqc(\riem^P) \,: \,  f_i(\disk)^{\mathrm{cl}}\subset K_i^{\mathrm{int}}, i=1,\ldots,n \}  \]
 is open in $\Oqc(\riem^P)$.  Letting
 \[  W = \{ (f_1 - p_1,\ldots,f_n -p_n) \,:\, (f_1,\ldots,f_n) \in V \}  \subseteq {\Oqc}(n) \]
 and applying \cite[Theorem 3.11]{RSnonoverlapping} with
 coordinates $\zeta_i(z)=z - p_i$ the map
 \begin{align} \label{eq:local_coord_G}
  G:W & \longrightarrow V  \nonumber \\
  (g_1,\ldots,g_n) & \longmapsto (g_1 + p_1,\ldots,g_n + p_n)
 \end{align}
 is a biholomorphism.
 \end{remark}

 Let
 \begin{equation} \label{eq:Bn} \mathfrak{B}(n) = \left\{ T: \bigoplus^n A^2(\disk^-) \longrightarrow \bigoplus^n A^2(\disk^+)
    \,:\,  \| T \| <\infty \right\}.
 \end{equation}
 Recall that we are using the Pythagorean norm on $\bigoplus^n A^2(\disk^\pm)$.
 \begin{remark} \label{re:block_holomorphicity}  In addressing holomorphic dependence of the  Grunsky operator on the rigging $\bm{f}$ below, we will need the following elementary observation.  Let
 \[  T: \bigoplus^n A^2(\disk^-) \rightarrow \bigoplus^n A^2(\disk^+), \]
 be a linear operator and let $T_{jk}:A^2(\disk^-) \rightarrow A^2(\disk^+)$ be its blocks.
 Using the inequality $\sup_{1\leq k\leq n}|a_k|\leq\sqrt{ \sum_{k=1}^n |a_k|^2} \leq \sum_{k=1}^n |a_k|$, we obtain for $h = (h_1,\ldots,h_n)$ the inequality
 \[  \| T \| \leq \sup_{h \in \bigoplus^n A^2(\disk^-), \|h \| \leq 1}
 \sum_{k=1}^n \left\| \sum_{j=1}^n T_{jk} h_j \right\|_{A^2(\disk^+)}
   \leq \sup_{\|h_j \| \leq 1, j=1,\ldots,n}   \sum_{k=1}^n  \sum_{j=1}^n  \left\| T_{jk} h_j \right\|_{A^2(\disk^+)}.  \]

   Therefore, to show that $T(t)$ is G\^ateaux holomorphic at $t=0$,
   where $t$ is a complex parameter, it is enough to show that
   \[  \lim_{t \rightarrow 0}   \left\| t^{-1} \left( T_{jk}(t) - T_{jk}(0) - tB_{jk} \right)  \right\| = 0  \]
   for some $B:\bigoplus^n A^2(\disk^-) \rightarrow \bigoplus^n A^2(\disk^+)$ with blocks $B_{jk}$.
 \end{remark}

 \begin{theorem}  \label{th:holomorphicity_Grunsky_comp} Fix distinct points $p_1,\ldots,p_n \in \mathbb{C}$ and
  let $K_i$ be non-intersecting compact sets such that $p_i$ are
  in $K_i^{\mathrm{int}}$ for $i=1,\ldots,n$.
  Let $N = N_1 \times \cdots \times N_n$ where $N_i \subseteq \mathbb{C}$ are open neighbourhoods of
  $0$ such that the sets $K_i + z_i$ are non-intersecting for all $(z_1,\ldots,z_n) \in N$.
  Let $W= \{ (g_1,\ldots,g_n) \in \Oqc(n) \,:\, g_i(\disk)^{\mathrm{cl}} + p_i \subset K_i^{\mathrm{int}} \}$.
  The map
  \begin{align*}
   H : W    \times N & \longrightarrow \mathfrak{B}(n) \\
   (g_1,\ldots,g_n,a_1,\ldots,a_n) & \longmapsto \hatGrfull(g_1+a_1 + p_1,\ldots,g_n+a_n + p_n) \\
  \end{align*}
  is holomorphic.
 \end{theorem}

 \begin{proof}

   By \cite[p 198]{Chae} it is enough to show that $H$ is locally bounded and G\^ateaux
   holomorphic.  By Theorem \ref{th:Grunsky_bounded} $\hatGrfullf$ is bounded, so only
   G\^ateaux holomorphicity remains.   By Hartogs' theorem in the Banach space setting \cite{Mujica}
   it is enough to prove holomorphicity on $W$ and $N$ separately. Since $W$ is a subset of $\Oqc(n) = \oplus^n \Oqc$ we further reduce the problem to proving holomorphicity on the individual copies of $\Oqc$. Recall that the complex structure on $\Oqc$ is given by the pull-back of the
   complex structure on $\mathbb{C} \oplus A_1^\infty(\disk^+)$ under
   $g \mapsto (g'(0),g''/g')$ (see equation (\ref{eq:chi_definition})). So holomorphicity on $W$ has been finally reduced to G\^ateaux holomorphicity separately on $\mathbb{C}$ and $A_1^\infty(\disk^+)$.

Note that Remark \ref{re:block_holomorphicity} yields that the holomorphicity of $\hatGr$ follows from the holomorphicity of its blocks $\hatGr_{kl}$, for $k,l = 1,\ldots n$. Recall that the block $\hatGr_{kl}(\bm{f})$ is only a function of $f_k$ and $f_l$.

   We first look at the diagonal components $\hatGr_{ii}$. For fixed $a_1,\ldots,a_n \in N$, holomorphic dependence
   of $\hatGr_{ii}$ on $A_1^\infty(\disk^+)$ is due to Takhtajan and Teo
   \cite[Theorem B.1 p 109]{TakhtajanTeo_memoirs} (note that there they use the integral formula of
   Theorem \ref{th:integral_representation} as the definition of the operator).  Let $f_i = g_i + a_i + p_i$. Since
   $\hatGr_{ii}$ is invariant under $f_i \mapsto c f_i$ for $c \neq 0$ it is independent
   of $g'(0) = f'(0)$, and so
    $\hatGr_{ii}$ is holomorphic on $\mathbb{C}$.  For holomorphicity on $N$, one needs only to observe that $\hatGr_{ii}$ are independent of $(a_1,\ldots,a_n) \in N$.

 Now we prove that the off-diagonal components of $\hatGrfullf$ are holomorphic.
 First we fix $(a_1,\ldots,a_n) \in N$ and prove G\^ateaux holomorphicity on W. Fix $j \in \{1,\ldots,n \}$. We will prove G\^ateaux holomorphicity on the $j$th copy of $\Oqc$. This requires only looking at the blocks $\hatGr_{ji}$ and $\hatGr_{ij}$ for $i \neq j$.

Fix $(g_1^0, \ldots, g_n^0) \in W,$ and consider the complex lines $(q(t), \psi^t)  \in  \mathbb{C} \oplus A^\infty_1(\disk^+)$, where $\psi^t =(g^0_j)''/(g^0_j)'+ t \phi$ for some $\phi \in A^\infty_1(\disk^+)$ and $q(t) = (g_j^0)'(0) + c\,t$ for some $c \in \mathbb{C}$.
 Now define the curve $g^t_j \in \Oqc$ to be the solution of the differential equation $(g^t_j)''/(g^t_j)' = \psi^t$
  with initial conditions $g^t_j(0)=0$ and $(g^t_j)'(0)=q(t)$.   That is, the curve $g^t_j$
  corresponds to the above complex line under
  the map $\chi$ defined in (\ref{eq:chi_definition}).
  Since $\chi(\Oqc) \subseteq
  \mathbb{C} \oplus A^\infty_1(\disk^+)$ is open, there is an $r>0$ such that $g^t_j \in \Oqc$ for all $|t|<r$.

 Let $f_i^0 = g_i^0 + a_i + p_i$ for $i = 1,\ldots n$ and let $f_j^t = g_j^t + a_i + p_i$.
 Using Theorem \ref{th:integral_representation}, we now prove G\^ateaux holomorphicity by proving that for all $i\neq j$, $t \mapsto \hatGr^{t}_{ji}$  and $t \mapsto \hatGr^{t}_{ij}$ are holomorphic in a
neighborhood of $t=0$ in $\mathbb{C}$, where
   \[  \hatGr^t_{ji}h(z) =  \hatGr_{ji}(f_j^t, f_i^0) h(z) =  \frac{1}{\pi} \iint_{\mathbb{D}^+}  \frac{(f^0_i)'(\zeta) (f_j^t)'(z)}
    {(f^0_i(\zeta)-f_j^t(z))^2}  \, \hat{\mathbf{R}}h(\zeta)\, dA_\zeta     \]
  and
    \[  \hatGr^t_{ij} h(z) =  \hatGr_{ij}(f^0_i, f^t_j) h(z) =\frac{1}{\pi} \iint_{\mathbb{D}^+}  \frac{(f_j^t)'(\zeta) (f^0_i)'(z)}
    {(f_j^t(\zeta)-f^0_i(z))^2}  \, \hat{\mathbf{R}}h(\zeta)\,    dA_\zeta.     \]
    Let $L_1^t(z,\zeta):= \frac{(f^0_i)'(\zeta) (f_j^t)'(z)}
    {(f^0_i(\zeta)-f_j^t(z))^2}$ and $L_2^t(z,\zeta):=\frac{(f_j^t)'(\zeta) (f^0_i)'(z)}
    {(f_j^t(\zeta)-f^0_i(z))^2}.$

 To prove the holomorphicity of $\hatGr^{t}_{ji}$ we observe that, for fixed $z,\zeta \in\mathbb{D}^+$,  $L^t_1(z,\zeta)$ is a holomorphic function of $t$ in a  neighborhood of $0$. This follows from the fact that $f^t_j(z)$ is holomorphic
in $t$ for fixed $z$ (by construction; see
 \cite[p 287]{RSnonoverlapping} for an explicit expression).
 Now choose $\delta>0$ so that $\delta <r$.
Then using Cauchy's integral formula we have for all $|t|<\delta$
\begin{gather*}
L^{t}_{1} (z,\zeta)-L^{0}_{1} (z,\zeta) -t\left.
\frac{d}{dt}\right\vert_{t=0}L^{t}_{1} (z,\zeta)\\
=\frac{t^2}{2\pi i}
\oint_{|s|=\delta}\frac{L^{s}_{1} (z,\zeta)}{s^2 (s-t)}\,ds.
\end{gather*}

Setting  $\overline{u(z)}=\hat{\mathbf{R}} h$ and using the equality above together with the fact that $\hat{\mathbf{R}}$ is an isometry, we obtain
\begin{align*}
&\left\Vert \frac{{\hatGr}^{t}_{ji}-\hatGr^{0}_{ji}}{t} -
\left.\frac{d}{dt}\right\vert_{t=0}\hatGr^{t}_{ji}\right\Vert_{A^2(\disk^-) \rightarrow A^2(\disk^+)}\\
& \hspace{1cm} =\sup_{\Vert h \Vert_{{A^2(\disk^-)}}=1}\left\Vert \left(\frac{\hatGr^{t}_{ji}-\hatGr^{0}_{ji}(f)}{t} -
\left.\frac{d}{dt}\right\vert_{t=0}\hat{\Gr}^{t}_{ji,1}(f)\right)h\right\Vert_{A^2(\disk^+)}
\\&\hspace{1cm} = \frac{{|t|}}{2\pi}\sup_{\Vert u \Vert_{{A^2(\disk^+)}}=1}\left( \iint\limits_{\mathbb{D}^+}
\left|\iint\limits_{\mathbb{\disk^+}}
\left(\oint_{|s|=\delta}\frac{L^{s}_{1} (z,\zeta)}{s^2(s-t)}\, ds\right)\,\overline{u(\zeta)}\,
dA_\zeta \right|^2 dA_z\right)^{1/2}.
\end{align*}
By Fubini's theorem and the Cauchy-Schwarz inequality (in the contour integral), we see that
\begin{align*}
&\left\Vert \frac{{\hatGr}^{t}_{ji}-\hatGr^{0}_{ji}}{t} -
\left.\frac{d}{dt}\right\vert_{t=0}\hatGr^{t}_{ji}\right\Vert_{A^2(\disk^-) \rightarrow A^2(\disk^+)}\\
& \ =\frac{{ |t|}}{2\pi}\sup_{\Vert u \Vert_{{A^2(\disk^+)}}=1}\left( \iint\limits_{\mathbb{D}^+}
\left|\oint_{|s|=\delta}\frac{1}{s^2(s-t)}\,\left( \iint\limits_{\mathbb{\disk^+}}
L^{s}_{1} (z,\zeta) \overline{u(\zeta)}\, dA_\zeta \right)\,ds\,
\right|^2 dA_z\right)^{1/2}\\
& \ \leq \frac{{|t|}}{2\pi}\sup_{\Vert u \Vert_{{A^2(\disk^+)}}=1}\Biggl(
\iint\limits_{\mathbb{D}^+}\Biggl(\oint_{|s|=\delta}\frac{|ds|}{|s|^4 |s-t|^2}\Biggr)
\Biggl(\oint_{|s|=\delta}\left|\iint\limits_{\mathbb{D}^+}L^{s}_{1}(z,\zeta)\overline{u(\zeta)}\,
dA_\zeta \right|^2 |ds| \Biggr)dA_z\Biggr)^{1/2}\\
& \ \leq \frac{{ |t|}}{2\pi^2}\Biggl(\oint_{|s|=\delta}\frac{|ds|}{|s|^4|s-t|^2}\Biggr)^{1/2}
\sup_{\Vert u \Vert_{{A^2(\disk^+)}=1}} \left(\oint_{|s|=\delta}\left\Vert \iint\limits_{\mathbb{D}^+} L^{s}_{1}(z,\zeta) \,\overline{u(\zeta)}\,
dA_\zeta \right\Vert^2_{L^2 ( \mathbb{D}^+)}|ds|\right)^{1/2}.
\end{align*}

 Now we claim that for $|s|=\delta$ the operator with kernel  $L^{s}_{1}(z,\zeta)$ is bounded on $L^2 ( \mathbb{D}^+),$ with a norm that depends only on $\delta$.
To see this we observe that $f^0_j(\disk^+)$ and $f^0_i(\disk^+)$ have disjoint closures; furthermore, on any holomorphic curve through $f^0_j$ we can ensure that the closures of the images remain in fixed disjoint sets for sufficiently small $t$ \cite[Corollary 3.5]{RSnonoverlapping}.   As a consequence, for $|s|=\delta$ there is a constant $A_{\delta}>0$ such that $|f^0_{i}(\zeta)-f_{j}^{s}(z)|>A_{\delta}.$   Furthermore the image of $f^0_i$ and $f_j^s$ (for fixed $s$) are both bounded, so, $\| {f^0_i}' \|_{A^2(\disk^+)}$ and $\| {f_j^s}' \|_{A^2(\disk^+)}$ (for fixed $s$) are bounded.
 Again applying \cite[Corollary 3.5]{RSnonoverlapping} the image of $f_j^s$ is contained in a disk in $\mathbb{C}$ of radius independent of $s$ so the bound for $\| {f_j^s}' \|_{A^2(\disk^+)}$ can be chosen uniformly in $s$.
  Therefore, there exist constants $B_{\delta}>0$ and $C_{\delta}>0$ such that
\begin{equation}
\label{L2 bound on the kernel}
\left\{\iint\limits_{\mathbb{D}^+}\iint\limits_{\mathbb{D}^+}\left| L^{s}_{1}(z,\zeta)\right|^2 dA_z \, dA_\zeta \right\}^{1/2}\leq B_{\delta} \Vert (f^0_{i})'\Vert_{A^2 (\mathbb{D}^+)}\Vert (f^{s}_{j})' \Vert_{A^2 (\mathbb{D}^+)}\leq C_{\delta}.
\end{equation}

Now since the operator-norm of the integral operator with kernel $L^{s}_{1}(z,\zeta)$ (as a bounded linear operator from $L^2 (\mathbb{D}^+)$ to itself) is bounded by the left-hand side of \eqref{L2 bound on the kernel}, the claim follows.

Finally for $|t|<\delta$
\begin{gather*}
\left\Vert \frac{\hatGr^{t}_{ji}-\hatGr^{0}_{ji}}{t} -
\left.\frac{d}{dt}\right\vert_{t=0}\hatGr^{t}_{ji}\right\Vert_{A^2(\disk^-) \rightarrow A^2(\disk^+)}\\
\leq
C_{\delta}{|t|}\left(\oint_{|s|=\delta}\frac{|ds|}{|s|^4 |s-t|^2}
\oint_{|s|=\delta} \, \sup_{\Vert u \Vert_{A^2(\disk^+)}=1}\Vert u\Vert^{2}_{A^2 (\mathbb{D}^+)}|ds|\right)^{1/2}\\
\leq C'_{\delta}{|t|}\left(\oint_{|s|=\delta}\frac{2\pi \delta |ds|}{|s|^4 |s-t|^2}
\right)^{1/2}\\
\leq C''_{\delta_{1}}\frac{{ |t|}}{\delta_{1}-|t|},
\end{gather*}
which can be made as small as we like, provided $t$ is chosen small enough. This establishes the G\^ateaux holomorphicity of $\hatGr^{t}_{ji}$ in the first component of $W \times N$.  In the second component, the proof
above can be used in the same way: the integral kernel is holomorphic in $a_j$ under $f^0_j\mapsto
f^0_j +a_j$, so one only need to establish local boundedness.  By the hypotheses on $N$, $|f^0_i(\zeta) - f^0_j(z)|$ is still uniformly bounded and the integral estimate (\ref{L2 bound on the kernel})
continues to hold.  Proceeding as above we obtain holomorphicity in the second component.

The proof of holomorphicity of $\hatGr^{t}_{ij}$ is the same as the above, except that one replaces the $L^2(\mathbb{D}^+)$ boundedness of the integral operator with kernel $L_1^t(z,\zeta)$ with $L^2$ boundedness of the integral operator with kernel $ L_2^t(z,\zeta)$. This ends the proof of the theorem.
 \end{proof}

\end{subsection}
\end{section}
\begin{section}{Period map} \label{se:Teich_embedding}
\begin{subsection}{Fibration of the Teichm\"uller space of bordered surfaces}
In this section we recall some definitions of Teichm\"uller space and rigged Teichm\"uller
space.  We require some results of Radnell and Schippers \cite{RS_fiber} on
a fibration of Teichm\"uller space of surfaces with $n$ borders over the Teichm\"uller space of surfaces with $n$ punctures, which play a central role in the formulation of the period mapping and proof of its holomorphicity.\\

We say that a Riemann surface is a bordered surface of
type $(g,n)$ if it is a Riemann surface of genus $g$ with $n$ boundary
curves homeomorphic to $\mathbb{S}^1$.   More precisely, we
assume that the double of the
Riemann surface $\riem^D$ is of genus $2g+n-1$ and the ideal boundary $\partial \riem$
consists of $n$ closed analytic curves in $\riem^D$ each of which is homeomorphic
to $\mathbb{S}^1$ with respect to the topology inherited from $\riem^D$. We note that such a Riemann surface $\riem$ is a bordered surface in the sense of Ahlfors and Sario \cite{AhlforsSario}.  That is, there is an atlas of charts including boundary charts of the following form.
Any point of the boundary
is contained in a relatively open subset $U$ of the closure $\riem \cup \partial \riem$
such that there is a biholomorphism $\phi:U \rightarrow V$ where $V$ is a relatively open
subset of the upper half plane $\overline{\mathbb{H}} = \{ z \in \mathbb{C} \,:\, \text{Im}(z) \geq 0 \}$
and $\phi(U \cap \partial \riem)$ is an open interval on $\text{Im}(z)=0$.  We assume
that the transition functions $\phi_1 \circ \phi_2^{-1}$ of any pair of
charts are biholomorphic on their domain of definition.  In the case that
both are boundary charts, this means that maps $\phi_1 \circ \phi_2^{-1}$ have
biholomorphic extensions to an open set in $\mathbb{C}$ containing the original domain
of definition of $\phi_1 \circ \phi_2^{-1}$.\\

A quasiconformal map $f:\riem \rightarrow \riem_1$ between bordered Riemann surfaces of type $(g,n)$ must have an extension to the ideal boundary $\partial \riem$.  We will not distinguish this extension notationally from the
map on $\riem$.
We say that quasiconformal maps $f:\riem \rightarrow \riem$ and $g:\riem \rightarrow \riem$
are homotopic rel boundary if they are homotopic via a homotopy $H:[0,1] \times \riem \rightarrow \riem$ such that $H(t,z)=f(z)=g(z)$ for all $z \in \partial \riem$ and $t\in [0,1]$.\\

We now define the Teichm\"uller space of such a Riemann surface.
\begin{definition}
 Let $\riem$ be a Riemann surface whose universal cover is the unit disk.  The Teichm\"uller
 space of $\riem$ is
 \[  T(\riem) = \{ (\riem,f,\riem_1)  \} /\sim  \]
 where $f:\riem \rightarrow \riem_1$ is quasiconformal and
 $(\riem,f_1,\riem_1) \sim (\riem,f_2,\riem_2)$ if and only if there is a
 biholomorphism $\sigma:\riem_1 \rightarrow \riem_2$ such that $f_2^{-1} \circ \sigma \circ f_1$
 is homotopic to the identity rel boundary.  Denote the equivalence class of a
 triple $(\riem,f,\riem_1)$ by $[\riem,f,\riem_1]$.
\end{definition}

 In \cite{Radnell_thesis} Radnell defined a ``rigged Teichm\"uller space'' of a
 punctured surface, which was shown by Radnell and Schippers \cite{RadnellSchippers_monster} to be intermediate between the Teichm\"uller space of a bordered surface and
 that of the compact surface obtained by sewing disks on the boundary.  The rigged Teichm\"uller space and its relation to the usual Teichm\"uller space are instrumental in the proof of the main theorem.
 \begin{definition}
  Let $\riem^P_0$ be a compact surface with punctures $p_1,\ldots,p_n$.  The rigged Teichm\"uller space
  of $\riem^P_0$ is
  \[ \widetilde{T}(\riem^P_0) = \{ (\riem_0^P,F_1,\riem_1^P,\bm{f}) \,:\, F_1:\riem_0^P \rightarrow \riem_1^P \ \text{quasiconformal}, \
     \bm{f} \in \Oqc(\riem_1^P) \} / \sim  \]
  where $\sim$ is an equivalence relation defined by
  $(\riem_0^P,F_1,\riem_1^P,\bm{f}) \sim (\riem_0^P,F_2,\riem_2^P,\bm{g})$ whenever there is a conformal map
  $\sigma:\riem_1^P \rightarrow \riem_2^P$ preserving the punctures and their order such that
  $F_2^{-1} \circ \sigma \circ F_1$ is homotopic to the identity (in such a way that the homotopy is
  constant on the punctures) and $g_i = \sigma \circ f_i$ for all $i=1,\ldots,n$.
 \end{definition}

 There is in general a holomorphic fibration of the Teichm\"uller space of a bordered surface
 over the rigged Teichm\"uller space of a punctured surface.  We need this in the special case
 that $\riem_0$ is $\sphere$ minus disks.
 Fix a collection of disks $D_i = \{ z\,:\, |z-p_i|<r_i \}$, $i=1,\ldots,n$, such that $D_i^{\mathrm{cl}} \cap D_j^{\mathrm{cl}}$
 is empty whenever $i \neq j$.   Set $\Sigma_0 = \sphere \backslash \cup_{i=1}^n D_i^{\mathrm{cl}}$
 and $\Sigma_0^P = \sphere \backslash \{ p_1,\ldots,p_n \}$.
 Finally, fix $\bm{\tau}=(\tau_1,\ldots,\tau_n)$ where for each $i$ the map
 $\tau_i:\disk^+ \rightarrow D_i$ is a conformal bijection such that $\tau_i(0)=p_i$.

  Now let $[\riem_0,F_1,\riem_1] \in T(\riem_0)$.  Let $\mu(F_1)$ be the Beltrami differential
 of $F_1$ on $\riem_0$.  Extend $\mu(F_1)$ to a Beltrami differential on $\riem_0^P$  by setting
 \[  \mu^P(z) =
 \begin{cases}
 \mu(F_1)(z) & \text{for } z \in \riem_0 \\ 0 & \text{for } z \in \sphere \backslash \riem_0 . \\
    \end{cases} \]
 Let $F_1^P: \riem_0^P \rightarrow \sphere$ be a quasiconformal map with
 dilatation $\mu^P$.  Since $F_1^P$ is quasiconformal
 it has a unique continuous (in fact quasiconformal) extension to $\sphere$;
 we will use the same notation for the extension.

 We now define the fibration maps
 \begin{align*}
  \mathcal{P}:T(\riem_0) & \longrightarrow \widetilde{T}(\riem_0^P) \\
  [\riem_0,F_1,\riem_1] & \longmapsto \left[\riem_0^P,F_1^P,\riem_1^P,\left( F^P_1 \circ \tau_1,\ldots,F^P_1 \circ \tau_n \right) \right]
 \end{align*}
 and
 \begin{align*}
   \mathcal{F}:\widetilde{T}(\riem_0^P) & \longrightarrow T(\riem^P) \\
   [\riem_0^P,F_1^P,\riem_1^P,\bm{f}] & \longmapsto [\riem_0^P,F_1^P,\riem_1^P].
 \end{align*}

 We also require some results on the modular group; proofs and details can be found in \cite{RadnellSchippers_monster}.
 The modular group $\text{PModI}(\riem_0)$
 consists of the set of quasiconformal self-maps of $\riem_0$ which are the identity on $\partial \riem_0$,
 modulo homotopy rel boundary.    The ``P'' in ``PMod''
 stands for ``pure'',
 which signifies that the self-maps fix the ordering of the boundary components.
 Given a quasiconformal $\rho:\riem_0 \rightarrow \riem_0$ fixing the boundary, denote
 its equivalence class by $[\rho]$.   The modular group
 $\text{PModI}(\riem_0)$ acts on
 of $T(\riem_0)$ via
 \[  [\rho]^*[\riem_0,F_1,\riem_1] = [\riem_0,F_1 \circ \rho^{-1},\riem_1].  \]

 Let $\text{DB}$ be the subgroup of $\text{PModI}$ generated by Dehn twists around the boundary curves of $\riem_0$.
 It was proven in \cite[Theorem 5.6]{RadnellSchippers_monster} that
 \begin{theorem} \label{th:kernel_of_P}
  $\mathcal{P}(p)=\mathcal{P}(q)$ if and only if there is a $[\rho] \in \operatorname{DB}$ such
  that $[\rho]^* p = q$.
 \end{theorem}
 Furthermore by \cite[Corollary 6.2, Corollary 5.1]{RadnellSchippers_monster}
 \begin{theorem}  \label{th:FandPholomorphic}
  $\mathcal{F}$ and $\mathcal{P}$ are holomorphic.
 \end{theorem}

 Finally, we need one further result. Its statement is technical, but it is quite
 powerful for proving holomorphicity in situations which involve conformal welding, either implicitly
 or explicitly.  Here, welding is implicit in the extension of the Beltrami differentials
 by $0$ to the caps.   The general result
 can be found in \cite{RS_fiber}; we specialize
 to the situation that $\riem_0^P$ is a punctured sphere.  Some conditions relating to the
 normalization are added, which do not follow directly from the statement of the theorem in \cite{RS_fiber}.    For this reason we include a brief proof.
 \begin{theorem} \label{th:riggedTeich_coord}
   Fix $n > 3$. Let $\riem_0^P = \sphere \backslash \{p_1,\ldots,p_n \}$ for points $p_1,\ldots,p_n \in \sphere$.
   Let $d$ be the
 dimension of $T(\riem_0^P)$.  Fix any $p = [\riem_0^P,F_*^P,\riem_*^P] \in
 T(\riem_0^P)$ and let $(\riem_0^P,F_*^P,\riem_*^P)$ be the unique representative
 such that $\riem_*^P$ is a sphere with punctures $(q_1,q_2,q_3,\ldots,q_n)$
 where $q_i=p_i$ for $i=1,2,3$.
 Let $\bm{f} = (f_1,\ldots,f_n) \in \Oqc(\riem_*^P)$ be a rigging on $\riem_*^P$ and let $K_i$ be compact, non-overlapping sets on $\riem_*^P$ containing $p_i$
   in their interiors, and let $V$ be as in Remark \ref{re:Oqccoords}.

   There is
 an open set $N \subseteq \mathbb{C}^d$ containing $0$
 and a map $\nu:N \times \sphere \rightarrow \sphere$ such that
   \begin{enumerate}
    \item $\nu_\epsilon$ fixes $p_1$, $p_2$ and $p_3$ (where $\nu_\epsilon(z)=\nu(\epsilon,z)$),
    \item for fixed $\epsilon$, $\nu(\epsilon,z)$ is quasiconformal on $\sphere$
     and holomorphic on $\cup_{i=1}^n K_i$ (that is, one-to-one and meromorphic),
    \item $\nu(\epsilon,z)$ is holomorphic in $\epsilon$ for any fixed $z$, and
    \item denoting $\nu_{\epsilon}(z)=\nu(\epsilon,z)$, the map
     $\epsilon \mapsto [\riem_0,\nu_\epsilon \circ F_*^P,\nu_{\epsilon}(\riem_*^P)]$
     is a local biholomorphic coordinate system on $T(\riem_0^P)$ onto a neighbourhood
     of $p$.
   \end{enumerate}

 Furthermore, for this map $\nu_\epsilon$,
  \begin{enumerate} \setcounter{enumi}{4}
   \item the map
   \begin{align*}
    \Psi: N \times V & \longrightarrow \widetilde{T}(\riem_0^P) \\
    (\epsilon, \bm{f}) & \longmapsto [\riem_0^P,\nu_\epsilon \circ F_*^P, \nu_\epsilon(\riem_*^P),\nu_\epsilon \circ \bm{f}]
   \end{align*}
   is a local biholomorphic coordinate system on $\widetilde{T}(\riem_0^P)$.
  \end{enumerate}
 \end{theorem}
 \begin{proof}
  By a result of F. Gardiner \cite{Gardiner} (see also \cite[Theorem 4.3.2]{Nagbook}) there is a quasiconformal map
  $\nu_\epsilon$ on $\riem_*^P$ which is quasiconformal on $\sphere$, and holomorphic
  on $\cup_{i=1}^n K_i$ such that $\epsilon \rightarrow [\riem_*^P,\nu_\epsilon, \nu_\epsilon(\riem_*^P)]$ form holomorphic coordinates for $T(\riem_*^P)$ in a neighbourhood of $[\riem_*^P,\text{Id},\riem_*^P]$.
  Note that in Gardiner's construction the Beltrami differential of $\nu_\epsilon$ is given
  explicitly and depends holomorphically on $\epsilon$; we will require this fact ahead.
  Since change of base point in Teichm\"uller space is biholomorphic, we obtain that
  $\epsilon \mapsto [\riem_0,\nu_\epsilon \circ F_*^P,\nu_{\epsilon}(\riem_*^P)]$ are coordinates
  on $T(\riem^P)$ for a neighbourhood of $[\riem_0, F_*^P,\riem_*^P]$
  (\cite[Theorem 2.17 and text immediately following]{RS_fiber}).  Thus we have that (2) and (4) hold.
  Claim (5) is stated explicitly and proved in \cite[proof of Theorem 4.1]{RS_fiber}.
  Note that there the map $\Psi$ is labelled $H$.

  Since $\nu_\epsilon(\riem_*^P)$ is quasiconformally equivalent to a punctured sphere,
  by the uniformization theorem it is biholomorphic to the punctured sphere.  Thus we may
  normalize $\nu_\epsilon$ so that $\nu_\epsilon$ is a map of the punctured sphere
  which fixes $q_i=p_i$ for $i=1,2,3$.  Thus we obtain property (1), and the normalization
  obviously does not affect property (2).  Since the normalization preserves the equivalence
  class in both $T(\riem^P)$ and $\tilde{T}(\riem^P)$, the maps in properties (4) and (5)
  are unchanged and thus (4) and (5) continue to hold.

  Finally, recall that the dilatation of $\nu_\epsilon$ depends holomorphically on $\epsilon$;
  property (3) thus is a classical property of solutions to the Beltrami
  equation with holomorphically varying dilatation, \cite[Theorem 1.2.11 p 38]{Nagbook}.
 \end{proof}

\end{subsection}
\begin{subsection}{Representation of Teichm\"uller space by Grunsky matrices}
 We return to the problem of defining the period mapping.
Assume that $n >3$ and recall the definitions of $\hatGrfull$ and $\mathfrak{B}(n)$ from equations (\ref{eq:hatGr}) and (\ref{eq:Bn})  respectively.  We define
  \begin{align} \label{eq:pitilde_definition} \begin{split}
   \tilde{\Pi} : \widetilde{T}(\riem_0^P) & \longrightarrow T(\riem_0^P) \times \mathfrak{B}(n)  \\
   \left( \riem_0^P,F_1^P,\riem_1^P, \bm{f} \right) & \longmapsto
   \left( \left[\riem_0^P,F_1^P,\riem_1^P\right], \hatGrfull(\bm{f}) \right).
  \end{split}\end{align}
  To see that this is well-defined, observe that if
  \[  (\riem_0^P,F_1^P,\riem_1^P,\bm{f}_1) \sim (\riem_0^P,F_2^P,\riem_2^P,\bm{f}_2)  \]
  then there is a M\"obius transformation $\sigma:\sphere \rightarrow \sphere$ taking the
  punctures of $\riem_1^P$ to those of $\riem_2^P$ and such that $\bm{f}_2 = \sigma  \circ \bm{f}_1$.
  Thus $[\riem_0^P,F_1^P,\riem_1^P]=[\riem_0^P,F_2^P,\riem_2^P]$ by the definition of Teichm\"uller
  equivalence, and $\hatGrfull(\bm{f}_1)
  = \hatGrfull(\bm{f}_2)$ by Theorem \ref{th:Grunsky_Mobius_invariant}.  Thus $\tilde{\Pi}$ is well-defined.

  Define also $\Pi= \tilde{\Pi} \circ \mathcal{P}$.  In that case $\Pi$ is given by
 \begin{align}  \label{eq:pi_definition} \begin{split}
  \Pi: T(\riem_0) & \longrightarrow T(\riem_0^P) \times
  \mathfrak{B}(n) \\
  [\riem_0,F_1,\riem_1] &  \longmapsto \left( \left[\riem_0^P,F_1^P,\riem_1^P \right] ,
  \hatGrfull(\bm{f}) \right)
 \end{split} \end{align}
 where
 \[  \bm{f} = \left(  F_1^P \circ \tau_1,\ldots, F_1^P \circ \tau_n \right)  \]
 and $F_1^P$ is determined from $F_1$ via extending the Beltrami differential of $F_1$
 by zero on the caps, as specified in
 the previous section.  Since $\mathcal{P}$ is well-defined \cite{RadnellSchippers_monster}
 and $\tilde{\Pi}$ is well-defined, so is $\Pi$.
 Denote the two components of $\Pi$ by $\Pi_1:T(\riem_0) \rightarrow T(\riem^P_0)$ and
 $\Pi_2:T(\riem_0) \rightarrow \mathfrak{B}(n)$, and similarly for $\tilde{\Pi}$.

 If $n=1$, $n=2$ or $n=3$, the Teichm\"uller space of $\Sigma_0^P$ reduces to a point.
 In those cases, we
 define $\Pi$ and $\tilde{\Pi}$ as maps into $\mathfrak{B}(n)$:
   \begin{align} \label{eq:pitilde_definition_23} \begin{split}
   \tilde{\Pi} : \widetilde{T}(\riem_0^P) & \longrightarrow \mathfrak{B}(n) \\
   \left( \riem_0^P,F_1^P,\riem_1^P, \bm{f} \right) & \longmapsto
     \hatGrfull(\bm{f})
  \end{split}\end{align}
  and
   \begin{align}  \label{eq:pi_definition_23} \begin{split}
  \Pi: T(\riem_0) & \longrightarrow \mathfrak{B}(n)  \\
  [\riem_0,F_1,\riem_1] &  \longmapsto  \hatGrfull(\bm{f}).
 \end{split} \end{align}
 The case that $n=1$ was considered and shown to be holomorphic
 by Takhtajan and Teo \cite{TakhtajanTeo_memoirs}.
 \begin{remark}  \label{re:special_case_equiv_rel}
   In the cases that $n$ is equal to $1$, $2$, or $3$, the equivalence relation
   on $\widetilde{T}(\riem_0^P)$ says that two elements $\left(\riem_0^P,F_1^P,\riem_1^P, \bm{f} \right)$ and $\left(\riem_0^P,F_2^P,\riem_2^P, \bm{g} \right)$ are equivalent
   if and only if there is some conformal map $\sigma:\riem_1^P \rightarrow \riem_2^P$ such that
   $\sigma \circ f_i = g_i$ for $i=1,\ldots,n$.

 \end{remark}
 \begin{remark}  It is clear that $\Pi$ and $\tilde{\Pi}$ depend on $\bm{\tau}$.
 \end{remark}

 In order to prove the main theorem, we require a technical lemma.  Recall that the complex structure on $\Oqc$ is induced by $\mathbb{C} \oplus A_1^\infty(\disk^+)$
\begin{lemma}
\label{le:CompHoloOqc_manyvar}
 Let $E$ be an open subset of $\mathbb{C}$ containing $0$ and $\Delta$ an open
 subset of $\mathbb{C}$. Let $M: \Delta \times E \to \mathbb{C}$ be a map which is holomorphic in both variables and injective in the second variable and let $M_{\epsilon}(z) = M(\epsilon, z)$. Let $\psi \in \Oqc$ satisfy ${\psi(\disk^+)}^{\mathrm{cl}} \subseteq E$.
 Then the map $Q: \Delta \mapsto \Oqc$ defined by $Q(\epsilon) = M_\epsilon \circ \psi$ is holomorphic $($in $\epsilon$$)$.
 \end{lemma}
 \begin{proof}
 Define $\mathcal{A}(f) = f''/f'$.
  We need to show that for fixed $\psi$, $\mathcal{A}(M_\epsilon \circ \psi)$ and $(M_\epsilon \circ \psi)'(0)$ are holomorphic
  in $\epsilon$.  The second claim follows from the fact that the $z$-derivatives of  all orders of $M_\epsilon$ are holomorphic in $\epsilon$ for fixed $z$.

  To prove holomorphicity of $\epsilon \mapsto \mathcal{A}(M_\epsilon \circ \psi)$,
  it is enough to show weak holomorphicity and local boundedness in the $A_1^\infty(\disk^+)$ norm \cite{GrosseErdmann}; that is, to show  local boundedness and that for some set of separating continuous
  functionals $\{\alpha\}$ in the dual of $A_1^\infty(\disk^+)$, $\alpha \circ \mathcal{A}(M_\epsilon \circ \psi)$ is holomorphic for all $\alpha$.  Let $e_z$ be the
  point evaluation function $e_z \psi = \psi(z)$.  These are continuous on $A_1^\infty(\disk^+)$
  and obviously separating on any open set.
  Since
  \begin{equation} \label{eq:lemma_hol_temp}
    \mathcal{A}(M_\epsilon \circ \psi)= \mathcal{A}(M_\epsilon) \circ \psi \cdot \psi' + \mathcal{A}(\psi)
  \end{equation}
  clearly $e_z(\mathcal{A}(M_\epsilon \circ f))$ is holomorphic in $\epsilon$.

  Next, let $F \subseteq E$ be a simply connected open set such that ${\psi(\disk^+)}^{\mathrm{cl}} \subset F$. Let $\lambda_F(z)^2$ denote the hyperbolic line element on $F$ (that is, $\lambda_F(z) = |g'(z)|/(1-|g(z)|^2)$ for any biholomorphism $g:F \rightarrow \disk^+$).  Since $\psi(\disk^+)^{\mathrm{cl}}$ is compactly contained in $F$,  $1/\lambda_F(\psi(z))$ is bounded below on $\disk^+$.  Thus by the Schwarz lemma, there
  is a constant $C$ such that
  \begin{equation} \label{eq:temp_Schwarz}
    (1-|z|^2)|\psi'(z)| \leq \frac{1}{\lambda_F(\psi(z))} \leq C
  \end{equation}
  for all $z \in \disk^+$.

  It remains to show that $\mathcal{A}(M_\epsilon \circ \psi)$.   Equality  (\ref{eq:lemma_hol_temp}) yields that for any fixed $\epsilon$
  \[  \left\| \mathcal{A}(M_\epsilon \circ \psi)\right\|_{A_1^\infty(\disk^+)}
     \leq  \sup_{z \in \disk^+} \left| \mathcal{A}(M_\epsilon) \circ \psi(z)
       \right| \sup_{z \in \disk^+} (1-|z|^2) |\psi'(z)|
       +  \sup_{z \in \disk^+} (1-|z|^2) |\mathcal{A}(\psi)(z)|.   \]
   Since $\mathcal{A}(M_\epsilon)$ is jointly holomorphic in $\epsilon$ and $z$
and ${\psi(\mathbb{D})}^{\mathrm{cl}} \subseteq E$ for any fixed $\epsilon_0$, there is a compact set $D$ containing $\epsilon_0$ such that $|\mathcal{A}(M_\epsilon)|$ is bounded on $\psi(\disk^+)$ by a constant independent of $\epsilon \in D$.  Using (\ref{eq:temp_Schwarz}) and the fact that  $\mathcal{A}(\psi)$ is in $A_1^\infty(\disk^+)$ we obtain that $\mathcal{A}(M_\epsilon \circ \psi)$ is locally bounded, which completes the proof.
 \end{proof}

  We now prove the main theorem.
 \begin{theorem}  $\Pi$ and $\tilde{\Pi}$ are holomorphic.
 \end{theorem}
 \begin{proof}
   Since $\mathcal{P}$ and $\tilde{\Pi}_1= \mathcal{F}$ are holomorphic by Theorem \ref{th:FandPholomorphic},
   it suffices to show that $\tilde{\Pi}_2$ is holomorphic.

    Fix an arbitrary point $[\riem_0,F_*^P,\riem_*^P,\bm{f}^0] \in \widetilde{T}(\riem_0)$.
   We will show that $\tilde{\Pi}$ is holomorphic at this point.
   Choose the representative $\riem_*^P = \sphere \backslash \{p_1,\ldots,p_n \}$, and
   $\bm{f}^0=(f_1^0,\ldots,f_n^0) \in \Oqc(\riem_*^P)$.
    Let $K_i$, $W$,
   $V$, and $G$ be as in Remark \ref{re:Oqccoords} and choose $N$ as in Theorem \ref{th:holomorphicity_Grunsky_comp}.  If $n>3$,
   by Theorem \ref{th:riggedTeich_coord} and Remark \ref{re:Oqccoords}
   it is enough to show that $\widetilde{\Pi}_2 \circ \Psi \circ \left(\text{Id} \times G \right)$ is holomorphic,
   where $\text{Id}$ is the identity on $N$.   In the cases that $n=2$ or $n=3$, it automatically reduces to this.
   The explicit formula is
   \[  \widetilde{\Pi}_2 \circ \Psi \circ \left(\text{Id} \times G \right) (\epsilon, g_1,\ldots,g_n)
      = \hatGrfull(\nu_\epsilon(g_1 + p_1), \ldots, \nu_\epsilon(g_n + p_n)).  \]

   By Hartog's theorem \cite{Mujica} it is enough to show separate holomorphicity in
   $\epsilon$ and in $\Oqc(n)$.

   First we fix $\epsilon = 0$.  In that case, we have that
   \[  \widetilde{\Pi}_2 \circ \Psi \circ \left(\text{Id} \times G \right) (0, g_1,\ldots,g_n)
      = \hatGrfull(g_1 + p_1, \ldots, g_n + p_n).  \]
   This is holomorphic in $\Oqc(n)$ by applying Theorem \ref{th:holomorphicity_Grunsky_comp} with fixed
   $(a_1,\ldots,a_n)$.

   Now fix $(g^0_1,\ldots,g^0_n) = G^{-1}(\bm{f}^0) = (f_1^0-p_1,\ldots,f_n^0-p_n)$ and vary $\epsilon$.  In this case we have
   \[  \widetilde{\Pi}_2 \circ \Psi \circ \left(\text{Id} \times G \right) (\epsilon, g_1^0,\ldots,g_n^0)
      = \hatGrfull(\nu_\epsilon(f^0_1), \ldots, \nu_\epsilon(f^0_n)).   \]
   If we set $\hat{\nu}_i(\epsilon,z) = \nu_\epsilon(z)-\nu_\epsilon(p_i)$ then we can write
   \[   \hatGrfull(\nu_\epsilon(f_1^0), \ldots, \nu_\epsilon(f_n^0))
     = H(\hat{\nu}_1(\epsilon, f_1^0),\ldots,\hat{\nu}_n(\epsilon, f_n^0),\nu_\epsilon(p_1),\ldots,\nu_\epsilon(p_n)).  \]
     Where $H$ is defined in Theorem \ref{th:holomorphicity_Grunsky_comp}.
   Now by Theorem \ref{th:riggedTeich_coord} $\epsilon \mapsto \nu_\epsilon(p_i)$ is holomorphic
   in $\epsilon$, and by Lemma \ref{le:CompHoloOqc_manyvar} combined with Hartogs' theorem
   on separate holomorpicity in finitely many variables,  the map from $N$ to $\Oqc$ given by
   \[  \epsilon \longmapsto \hat{\nu}_i(\epsilon,f_i^0)  \]
   is holomorphic.  This together with Theorem \ref{th:holomorphicity_Grunsky_comp}
    shows that $\widetilde{\Pi}_2 \circ \Psi \circ \left(\text{Id} \times G \right)$
   is holomorphic at $(0,f_1^0,\ldots,f_n^0)$.  Since the point was arbitrary this completes the
   proof.
 \end{proof}

 The next theorem shows that the map $\tilde{\Pi}$ is injective, and $\Pi$ is nearly so.
 \begin{theorem}  $\tilde{\Pi}(p)=\tilde{\Pi}(q)$
  if and only if $p=q$.  $\Pi(p)=\Pi(q)$ if and only if there is a $[\rho] \in \operatorname{DB}$
  such that $\rho^*p=q$.
 \end{theorem}
 \begin{proof}  The second claim follows from the first by Theorem \ref{th:kernel_of_P}
 and the fact that $\tilde{\Pi} \circ \mathcal{P} = \Pi$.

 Now assume that $\tilde{\Pi}(p)=([\riem_0^P,F_1^P,\riem_1^P],S)$.
 For $S=\hatGrfullf \in \mathfrak{B}(n)$, the kernel functions of the block operators
 $\hatGr_{ij}$ are uniquely determined by $S$ for all $i,j$.  To see this, one need only apply
 the operator to the Bergman kernel for $A^2(\disk^+)$ in place of $\hat{\mathbf{R}} h$
 in Theorem \ref{th:integral_representation}.

 The proof proceeds in cases, depending on whether $n=1$, $n=2$, or $n \geq 3$.  If $n=1$,
 as already observed, the function
 \[   \frac{1}{(\zeta-z)^2} - \frac{f_1'(\zeta)f_1'(z)}{(f_1(\zeta)- f_1(z))^2}  \]
 is uniquely determined by $S$.  Now letting $\zeta \rightarrow z$, identity (3.7) in Bergman-Schiffer's paper
 \cite{BergmanSchiffer} yields that this quantity tends to one-sixth of the Schwarzian derivative of $f_1$.
 Thus the Schwarzian of $f_1$ is uniquely determined by $S$, and therefore $f_1$ is uniquely determined up to
 post-composition by a M\"obius transformation.  The claim now follows from  Remark \ref{re:special_case_equiv_rel}.

 Now assume that $n=2$, and that $\tilde{\Pi}(p)=([\riem_0^P,F_1^P,\riem_1^P],S)$
 is given.   Let $(\riem_0^P,F_1^P,\riem_1^P,\bm{f})$ be any representative of $p$, where
 $\bm{f}=(f_1,f_2)$.  By post-composing $\riem_1^P$, $F_1^P$, $f_1$ and $f_2$ simultaneously by
 a M\"obius transformation $\sigma$, we can assume that $f_1$ and $f_2$ are normalized
 so that $f_1(0)=0$, $f_1'(0)=1$ and $f_2(0)=1$ (any fixed value will do).  If it can be
 shown that this uniquely determines $f_1$ and $f_2$, then by Remark \ref{re:special_case_equiv_rel}
 it will follow that $\widetilde{\Pi}(p)=\widetilde{\Pi}(q)\Rightarrow p=q$.

 By the first paragraph of the proof, we have that the kernel function of $\text{Gr}_{12}(f)$
 \begin{equation} \label{eq:injective_proof_temp}
  \frac{f_1'(\zeta) f_2'(z)}{(f_1(\zeta)-f_2(z))^2}
 \end{equation}
 is uniquely determined by $\widetilde{\Pi}(p)$.  Setting $\zeta = z = 0$ yields that
 $f_2'(0)$ is uniquely determined.  Differentiating (\ref{eq:injective_proof_temp})
 with respect to $z$, we see that
 \begin{equation} \label{eq:injective_proof_temptwo}
  \frac{f_1'(\zeta) f_2''(z) \left( f_1(\zeta)-f_2(z) \right) + 2 f_1'(\zeta) f_2'(z)^2}
     {(f_1(\zeta)-f_2(z))^3}
 \end{equation}
 is uniquely determined, and setting $\zeta = z = 0$, one can also determine $f_2''(0)$ uniquely.
 The same argument applied to $\text{Gr}_{21}(f)$ shows that $f_1''(0)$ is determined
 uniquely, and applying the considerations in the first paragraph to $\text{Gr}_{11}(f)$ and $\text{Gr}_{22}(f)$
 shows that the Schwarzians of $f_1$ and $f_2$ are determined by $\widetilde{\Pi}(p)$.
 Since we have determined $f_i(0)$, $f_i'(0)$, $f_i''(0)$ for $i=1,2$, $f_i$'s
 are uniquely determined and the claim follows.

 Now we consider the case that $n \geq 3$; again assume $\tilde{\Pi}(p)=([\riem_0^P,F_1^P,\riem_1^P],S)$
 is given and $(\riem_0^P,F_1^P,\riem_1^P,\bm{f})$ be any representative of $p$.  By post-composing $\riem_1^P$, $F_1^P$, and $\bm{f}$ simultaneously by
 a M\"obius transformation $\sigma$, we can assume that $f_1, f_2$ and $f_3$ are normalized
 so that $f_1(0)=0$, $f_2(0)=1$ and $f_3(0)=-1$ (again, any fixed values will do).
 If $k \geq 3$, then the remaining values of the points $f_k(0)$ for $k \neq 1,2,3$
 are uniquely determined by the values of $F_1^P(p_k)$.  This is because these values are determined
 by the Teichm\"uller equivalence class $[\riem_0^P,F_1^P,\riem_1^P]$ up to post-composition
 by $\sigma$, and $\sigma$ is uniquely determined by the normalizations above.
 As in
 the $n=2$ case, by Remark \ref{re:special_case_equiv_rel} it is enough
 to show that $f_1,\ldots,f_n$ are now uniquely determined.

 Arguing as in the $n=2$ case, for any $i \neq j$, $\widetilde{\Pi}(p)$ uniquely determines
 \[  \frac{f_i'(\zeta) f_j'(z)}{(f_i(\zeta)-f_j(z))^2}.     \]
 Again setting $\zeta = z =0$, we see that for
 $i \neq j$, all pair-wise products $f_i'(0) f_j'(0)$ are uniquely determined.  By an easy algebraic argument, fixing any three
 pair-wise distinct values $i$, $j$, and $k$, the resulting three products
 $f_i'(0) f_j'(0)$, $f_j'(0) f_k'(0)$, and $f_k'(0) f_i'(0)$ uniquely determine $f_i'(0)$,
 $f_j'(0)$ and $f_k'(0)$.  Since $i,j,k$ are arbitrary we have shown that $f_i'(0)$ is
 determined uniquely for $i=1,\ldots, n$.

 By differentiating the kernels $\hatGr_{ij}$ twice with respect to $z$
 and setting $\zeta = z =0$  as in the $n=2$
 case, we uniquely determine $f_i''(0)$ for all $i=1,\ldots,n$.  Also once again, we have that
 $\text{Gr}_{ii}$ uniquely determines the Schwarzian of $f_i$ for all $i$.  Thus $f_i$'s
 are uniquely determined, and this completes the proof.
 \end{proof}
 \begin{remark}  One can of course compose the map $\Pi$ by the classical period map on $T(\riem^P)$
 to obtain a full embedding of $T(\riem)$ by a period mapping.
 \end{remark}
\end{subsection}
\end{section}

\end{document}